\documentclass{amsart}
\usepackage{amsmath, amssymb, amsthm, enumerate,amsfonts,latexsym,mathrsfs}
\usepackage{hyperref}
\newtheorem{theorem}{Theorem}[section]

\newtheorem{lemma}[theorem]{Lemma}

\newtheorem*{thmA}{Theorem~A}
\newtheorem*{thmB}{Theorem~B}
\newtheorem*{thmC}{Theorem~C}

\theoremstyle{definition}

\newtheorem*{example}{Example}

\newtheorem*{notation}{Notation}

\theoremstyle{remark}
\newtheorem{step}{\bf Step}

\newcommand{\Irr}{{\mathrm {Irr}}}
\newcommand{\cd}{{\mathrm {cd}}}
\newcommand{\Aut}{{\mathrm {Aut}}}

\newcommand{\Centralizer}{{\mathrm {C}}}

\newcommand{\PSL}{{\mathrm {PSL}}}

\newcommand{\PSU}{{\mathrm {PSU}}}

\newcommand{\PGL}{{\mathrm {PGL}}}

\begin{document}

\title[Prime graphs of finite groups]{Groups whose prime  graphs have no triangles}

\author{Hung P. Tong-Viet}
\email{Tongviet@ukzn.ac.za}
\address{School of Mathematics, Statistics and Computer Science\\
University of KwaZulu-Natal\\
Pietermaritzburg 3209, South Africa}

\subjclass[2000]{Primary 20C15; Secondary 05C25}

\date{\today}
\thanks{This research is supported by a Startup Research Fund from the College of Agriculture, Engineering and Science, the University of KwaZulu-Natal}
\keywords{character degrees; prime  graphs; triangles; trees; cycles}
\begin{abstract}
Let $G$ be a finite group and let $\cd(G)$ be the set of all complex irreducible character degrees
of $G.$ Let $\rho(G)$ be the set of all primes which divide some character degree of $G.$ The prime
graph $\Delta(G)$ attached to $G$ is a graph whose vertex set is $\rho(G)$ and there is an edge
between two distinct primes $u$ and $v$ if and only if the product $uv$ divides some character
degree of $G.$ In this paper, we show that if $G$ is a finite group whose prime  graph $\Delta(G)$
has no triangles, then $\Delta(G)$ has at most $5$ vertices. We also obtain a classification of all
finite graphs with $5$ vertices and having no triangles which can occur as prime graphs of some
finite groups. Finally, we show that the prime graph of a finite group can never be a cycle nor a
tree with at least $5$ vertices.
\end{abstract}
\maketitle
%%%%%%%%%%%%%%%%%%%%%%%%%%%%%%%%%%%%%%%%%%%%%%%%%%%%%%%%%%%%%%%%%%%%%%%%%%%%%%%%%%%%%%%%%%%%%%%%%%%%%%

\section{Introduction}
 Let $G$ be a finite group and let $\cd(G)$ be the set of all character degrees of $G,$ that is,
$\cd(G)=\{\chi(1)\:|\:\chi\in\Irr(G)\},$ where $\Irr(G)$ is the set of all complex irreducible
characters of $G.$ We write $\rho(G)$ to denote the set of all primes which divide some character
degree of $G.$ It is well known that the character degree set $\cd(G)$ can be used to obtain
information about the structure of  the group $G.$ For example, the celebrated Ito-Michler's
Theorem states that if a prime $p$ divides no character degree of a finite group $G,$ then $G$ has
a normal abelian Sylow $p$-subgroup. Another well known result due to J. Thompson says that if a
prime $p$ divides every nontrivial character degree of a group $G,$ then $G$ has a normal
$p$-complement. A useful way to study the character degree set of a finite group $G$ is to attach a
graph structure on  $\cd(G).$

For a finite group $G,$ there are several ways to define a graph structure on the set $\cd(G).$ The
{\bf prime  graph} $\Delta(G)$ of $G$ is a graph whose vertex set is $\rho(G)$ and there is an edge
between two distinct primes $r$ and $s$ if and only if the product $rs$ divides some character
degree of $G.$ This graph was first defined in \cite{Manz-Staszewski-Willems}  and has been studied
extensively since then. There is also another graph attached to the set $\cd(G)$ which we call the
{\bf degree graph} and denote by $\Gamma(G).$ This graph has $\cd(G)-\{1\}$ as the vertex set, and
there is an edge between two distinct nontrivial degrees $a,b\in\cd(G)$ if and only if $\gcd(a,b)$
is nontrivial. However the prime graph $\Delta(G)$ is used more often as it is compatible with
normal subgroups and factor groups in the sense that the prime graphs of those groups are subgraphs
of the prime graph $\Delta(G).$ This property is very useful for induction argument.

Usually, there is a closed connection between the two graphs $\Delta(G)$ and $\Gamma(G).$ For
instance, $\Delta(G)$ is disconnected if and only if $\Gamma(G)$ is. However, there are examples
showing that while $\Gamma(G)$ contains no triangles but $\Delta(G)$ has one and vice versa.
Indeed, if $G\cong 2\cdot \textrm{A}_6,$ then $\cd(G)=\{1,4,5,8,9,10\}$ and we can see that the
prime $2$ divides three distinct nontrivial degrees of $G,$ and thus these degrees form a triangle
in $\Gamma(G);$ however $\Delta(G)$ has no triangles. Now let $G\cong\PSL_2(29).$ Then
$\cd(G)=\{1,15,28,29,30\}.$ In this case, $\Delta(G)$ has a triangle but $\Gamma(G)$ has none. The
finite group $G$ whose degree graph $\Gamma(G)$ has no triangles has been investigated by Lewis and
White \cite{Lewis-White10}. They showed that for a nonsolvable group $G,$ $\Gamma(G)$ has no
triangles if and only if there is no primes which divides three distinct character degrees of $G.$
Furthermore, if $G$ is any finite group whose degree graph $\Gamma(G)$ has no triangles, then
$|\cd(G)|\leq 6.$ In this paper, we obtain the following result.

\begin{thmA}\label{Groups without triangles} If $G$ is any finite group whose prime graph $\Delta(G)$ has no
triangles, then $\Delta(G)$ has at most $5$ vertices.
\end{thmA}

The upper bound for $|\rho(G)|$ obtained in Theorem~A  is best possible as demonstrated in the
following example.

\begin{example}\label{example}

$(1)$ If $G\cong \PSL_2(2^6),$ then $|\rho(G)|=5,$ and $\Delta(G)$ has no triangles. The prime
graph $\Delta(\PSL_2(2^6))$ is isomorphic to the second graph in Figure A. Note that
$\Delta(\PSL_2(2^6))$ is disconnected with three connected components.

$(2)$ Let $G$ be a direct product of a simple group $H$ and a solvable group $K,$ where $H\cong
\textrm{A}_5$ or $\PSL_2(8)$ and $\Delta(K)$ has two connected components and two vertices such
that $\rho(K)\cap \rho(H)=\emptyset.$ Then $|\rho(G)|=5,$ $\Delta(G)$ is connected without
triangles and $\Delta(G)$ is isomorphic to the first graph in Figure A. To give an example of such
a group $K,$ let $L$ be the normalizer in $\PSU_3(23)$ of a Sylow $23$-subgroup of $\PSU_3(23).$
Then we can choose $K$ to be a Hall $\{11,23\}$-subgroup of $L.$ It follows that $K\cong
23^{1+2}:11$ and it is easy to check that $\cd(K)=\{1,11,23\},$ and so $K$ satisfies the required
conditions.
\end{example}

We next obtain a classification of finite graphs with $5$ vertices and having no triangles which
can occur as prime graphs of some finite groups. The structure of such finite groups is also
described in detail. It turns out that there are only two nonisomorphism types of such graphs and
they are given in Figure A. Notice that the existence of these graphs have been established in the
examples given above.

\begin{thmB} Let $G$ be a finite group such that $\Delta(G)$ has exactly $5$
vertices. If $\Delta(G)$ has no triangles, then the following hold.
\begin{enumerate}
\item If $\Delta(G)$ is disconnected, then $G\cong\PSL_2(2^f)\times A,$ where $A$ is abelian, $|\pi(2^f\pm 1)|=2$ and  $\Delta(G)$
is the second graph in Figure A,
\item If $\Delta(G)$ is connected, then $G=H\times K,$ where $H\cong \rm{A}_5$ or $\PSL_2(8),$
 $K$ is a solvable group such that $\Delta(K)$ has exactly two vertices and two connected
components and $\rho(H)\cap \rho(K)$ is empty. Furthermore, $\Delta(G)$ is the first graph in
Figure A.
\end{enumerate}
\end{thmB}
\setlength{\unitlength}{1.6mm}\vspace*{0.4cm}
\begin{center}
\begin{picture}(30,10)
% Graph I

% Dots
    \put(-5,10){\circle*{1}}
    \put(-5,0){\circle*{1}}
    \put(5,0){\circle*{1}}
    \put(5,10){\circle*{1}}
    \put(0,5){\circle*{1}}

% Lines
    \put(-5,10){\line(1,0){10}}
    \put(-5,0){\line(0,1){10}}
    \put(5,0){\line(-1,0){10}}
    \put(5,0){\line(0,1){10}}
    \put(0,5){\line(1,1){5}}
    \put(0,5){\line(-1,-1){5}}

%Graph II

% Dots
    \put(15,10){\circle*{1}}
    \put(15,0){\circle*{1}}
    \put(25,0){\circle*{1}}
    \put(25,10){\circle*{1}}
    \put(20,5){\circle*{1}}

% Lines
    \put(15,10){\line(1,0){10}}
    \put(15,0){\line(1,0){10}}
    \put(10,-5){\title{Figure A}}
    \end{picture}
\end{center}
\vspace*{0.8cm}

 It is worth mentioning that if $G$ is any finite group whose prime graph is the first graph in
Figure A, then the structure of $G$ is given in Case $(2)$ of Theorem~B.

One of the main questions in this area is to determine which finite graph could be or could not be
the prime graph of some finite group. For example, Moret\'{o} and Tiep \cite{Moreto-Tiep} showed
that an octagon cannot be a prime graph of any finite groups. Lewis and White \cite{Lewis-White12}
proved that a path with $4$ vertices cannot be a prime graph of any finite groups. Also, the
pentagon has been proved not to be a prime graph of any solvable group by M. Lewis \cite{Lewis03}.
However, it is easy to verify that any other cycles or trees with at most $4$ vertices can occur as
prime graphs of finite groups. Furthermore, it is proved in \cite{Lewis-Meng,Lewis-White12} that if
$\Delta(G)$ is a square, then $G$ must be a direct product and in particular is solvable. Recall
that a tree is a simple connected graph without any cycles. Notice that by definition, a cycle must
have at least $3$ vertices. As a consequence of our previous theorems, we show that these are in
fact the only cycles or trees which can be the prime graphs of some finite groups.

\begin{thmC}
Let $G$ be a finite group. If $\Delta(G)$ is a cycle or a tree, then $\Delta(G)$ has at most $4$
vertices.
\end{thmC}
We note that if the prime graph of a finite group is a cycle of length four, then the group must be
solvable. However if $\Delta(G)$ is a triangle, then $G$ need not be solvable. For instance, the
prime graph of the nonabelian simple group $\PSL_3(3)$ is a triangle. Another consequence of the
main theorems is that if $G$ is a finite group whose prime graph $\Delta(G)$ is a bipartite graph
$K_{m,n},$ where $1\leq m\leq n,$ then $|\rho(G)|=m+n\leq 5.$ Furthermore, if $m+n=5,$ then
$\Delta(G)$ is the first graph in Figure A; and if $m+n\leq 4,$ then all possibilities for $m$ and
$n$ can occur.

\begin{notation}
Throughout this paper, all groups are finite and all characters are complex characters. If $n\geq
1$ is an integer, then we denote the set of all prime divisors of $n$ by $\pi(n).$  If $G$ is a
group, then we write $\pi(G)$ instead of $\pi(|G|)$ for the set of all prime divisors of the order
of $G.$   If $N\unlhd G$ and $\theta\in\Irr(N),$ then the inertia group of $\theta$ in $G$ is
denoted by $I_G(\theta).$ We write $\Irr(G|\theta)$ for the set of all irreducible constituents of
$\theta^G.$ Recall that a group $G$ is said to be an {almost simple group} with socle $S$ if there
exists a nonabelian simple group $S$ such that $S\unlhd G\leq \Aut(S).$ For $\epsilon=\pm,$ we use
the convention that $\textrm{PSL}_n^\epsilon(q)$ is $\textrm{PSL}_n(q)$ if $\epsilon=+$ and
$\textrm{PSU}_n(q)$ if $\epsilon=-.$ The greatest common divisor of two integers $a$ and $b$ is
$\gcd(a,b).$ Other notation is quite standard.
\end{notation}

%%%%%%%%%%%%%%%% NOTATION %%%%%%%%%%%%%%%%%%%%%%%%%%%%

\section{Preliminaries}

In this section, we present some results that will be needed for the proofs of our theorems. We
first begin with a result due to P. P\'{a}lfy which gives a restriction on the prime  graphs of
solvable groups. This condition will be very useful in determining which graph could be the prime
graphs of solvable groups.

\begin{lemma}\emph{(P\'{a}lfy's Condition~\cite[Theorem]{Palfy}).}\label{Palfy condition} Let $G$ be a solvable group and
let  $\pi$ be a set of primes contained in $\rho(G)$ with $|\pi|=3.$ Then there exists an
irreducible character of $G$ with degree divisible by at least two primes from $\pi.$
\end{lemma}
The next result is an application of P\'{a}lfy's Condition. This gives a proof of Theorem~A for
solvable groups.
\begin{lemma}\label{Square or Triangle}  If $G$ is solvable and $\Delta(G)$ has no triangles, then $|\rho(G)|\leq 4.$
\end{lemma}

\begin{proof} Let $G$ be a solvable group. If $|\rho(G)|\leq 3,$ then we are done. Assume that $\Delta(G)$
has no triangles and $|\rho(G)|\geq 4.$ By \cite[Lemma~2.1]{Lewis-Meng}, $\Delta(G)$ is a square
and so $|\rho(G)|=4.$ Thus if $G$ is solvable and $\Delta(G)$ has no triangles, then $|\rho(G)|\leq
4.$
\end{proof}

The following number theoretic result due to Zsigmondy is very useful.
\begin{theorem}\emph{(Zsigmondy's Theorem~\cite{Zsigmondy}).}
Let $a\geq 2$ and $n\geq 2$ be integers. Then there exists a prime $\ell$ such that $\ell$ divides
$a^n-1$ but it does not divide $a^m-1$ for any $1\leq m<n$ unless
\begin{enumerate}
\item $n=6$ and $a=2$ or
\item $n=2$ and $a=2^r-1$ is a Mersenne prime, where $r$ is a prime.
\end{enumerate}

\end{theorem}
Such a prime $\ell$ is called a {\bf primitive prime divisor}. For fixed $a$ and $n,$ the smallest
primitive prime divisor of $a^n-1$ (if exists) is denoted by $\ell_n(a).$

The next lemma is an easy consequence of Zsigmondy's Theorem.

\begin{lemma}\label{2-powers} Let $f\geq 6$ be an integer. Suppose that
$f=nb,$ where $n\geq 3$ is a prime and $b\geq 1$ is an integer. Then ${(2^{2f}-1)}/{(2^{2b}-1)}$
cannot be a prime power.
\end{lemma}

\begin{proof}
By way of contradiction, assume that $(2^{2f}-1)/(2^{2b}-1)=r^m,$ where $r$ is a prime, and $m\geq
1$ is an integer. If $f=6,$ then $n=3$ and $b=2.$ Thus
${(2^{2f}-1)}/{(2^{2b}-1)}={(2^{12}-1)}/{(2^4-1)}=3\cdot 7\cdot 13$ is not a prime power. Hence we
can assume that $f>6$ and thus $2f>f>6.$ By Zsigmondy's Theorem, the primitive prime divisors
$\ell_{2f}(2)$ and $\ell_{f}(2)$ exist. Furthermore, these two primes  are distinct  and do not
divide $2^{2b}-1$ since $f>2b.$  Therefore, both $\ell_{2f}(2)$ and $\ell_{f}(2)$ must divide
$r^m,$ which is impossible.
\end{proof}

 The following result due to Gallagher will be used frequently.

 \begin{lemma}\emph{(Gallagher's Theorem~\cite[Corollary~6.17]{Isaacs}).}\label{Gallagher theorem}
 Let $G$ be a group and let $N\unlhd G.$ If $\theta\in\Irr(N)$ is extendible to
 $\theta_0\in\Irr(G),$ then the characters $\theta_0\lambda$ for $\lambda\in\Irr(G/N)$ are all of the irreducible
 constituents of $\theta^G.$  In particular, $\theta(1)\lambda(1)\in\cd(G)$ for all $\lambda\in\Irr(G/N).$
 \end{lemma}

Gallagher's Theorem is often used in combination with the following.

 \begin{lemma}\emph{(\cite[Theorem~11.7]{Isaacs}).}\label{Extension}
 Let $G$ be a group and let $N\unlhd G.$ Suppose that $\theta\in\Irr(N)$ is $G$-invariant and that the Schur multiplier of
 $G/N$ is trivial. Then $\theta$ is extendible to $\theta_0\in\Irr(G).$
 \end{lemma}

The next result is a consequence of the classification of subgroups of prime power index in
nonabelian simple groups due to Guralnick \cite{Guralnick}.
\begin{lemma}\label{power index and solvable}
Let $G$ be a nonabelian simple group and $H$ be a proper subgroup of $G.$ Suppose that $|G:H|=r^a$
for some prime $r.$ Then $H$ is either nonsolvable or  a nonabelian Hall subgroup of $G.$
\end{lemma}

\begin{proof}
Using the classification  of prime power index subgroups of finite simple groups due to Guralnick
\cite{Guralnick}, one of the following cases holds.

\begin{enumerate}
\item $(G,H)=(\textrm{A}_n,\textrm{A}_{n-1})$ with $n=r^a;$
\item $G\cong\PSL_n(q)$ and $H$ is the stabilizer of a line or a hyperplane and $r^a=(q^n-1)/(q-1),$ where $n$ is prime;
\item $(G,H,r^a)=(\PSL_2(11),\textrm{A}_5,11)$   or $(\PSU_4(2),2^4:\textrm{A}_5,3^3).$
\item $(G,H,r^a)=(\textrm{M}_{23},\textrm{M}_{22},23)$  or $(\textrm{M}_{11},\textrm{M}_{10},11);$
\end{enumerate}

Assume first that Case $(1)$ holds. Then $G\cong \textrm{A}_n,H\cong \textrm{A}_{n-1}$ and $n=r^a.$
If $n=5,$ then $H\cong \textrm{A}_4$ is obviously a nonabelian Hall subgroup of $G.$ Hence we can
assume that $n\geq 7.$ Then $H\cong \textrm{A}_{n-1}$ is nonsolvable since $n-1\geq 6.$

Assume next that Case $(2)$ holds. Assume that $n=2.$ Then $G\cong\PSL_2(q)$ with $q\geq 4,$ and
$H$ is a nonabelian group of order $q(q-1)/\gcd(2,q-1)$ with $r^a=q+1.$ Hence we only need to show
that $q+1$ is prime to $q(q-1)/\gcd(2,q-1).$ If $q$ is even, then $q+1$ is prime to $|H|=q(q-1)$
because $q+1$ is prime to both $q$ and $q-1.$ Now assume that $q$ is odd. Then $q+1=r^a$ is even
and so $r=2.$ Also, since $q+1=2^a$ where $q\geq 4,$ we deduce that $a\geq 3.$ It follows that
$(q-1)/2=2^{a-1}-1$ is odd. Hence $|H|=q(q-1)/2$ is odd. Therefore, $|H|$ is prime to $q+1=2^a$ as
required. Assume next that $n=3.$ Then $G\cong \PSL_3(q)$ and $H$ is the stabilizer of a line or a
hyperplane and $|G:H|=(q^3-1)/(q-1)=r^a.$ If $q=2,$ then $G\cong\PSL_3(2)$ and $H\cong
\textrm{S}_4$ with index $r^a=7.$ If $q=3,$ then $G\cong\PSL_3(3),H\cong 3^2:2\textrm{S}_4$ and
$|G:H|=13.$ In both cases, we see that $H$ is a nonabelian Hall subgroup of $G.$ Now assume that
$q\geq 4.$ It follows that $H$ possesses a section isomorphic to the nonabelian simple groups
$\PSL_2(q)$ and thus $H$ is nonsolvable. Finally, assume that $G\cong \PSL_n(q),$ with $n\geq 5,$
and $H$ is the stabilizer of a line or a hyperplane. It follows that $H$ possesses a section
isomorphic to $\PSL_{n-1}(q),$ which is nonsolvable as $n-1\geq 4.$ Thus $H$ is nonsolvable.

For Cases $(3)$ and $(4),$ we can see that $H$ is nonsolvable.
\end{proof}

Finally, by Ito-Michler's Theorem \cite[Theorem~5.4]{Michler}, we deduce that if $G$ is an almost
simple group, then $\rho(G)=\pi(G)$ as $G$ has no nontrivial normal abelian Sylow subgroups. This
fact will be used without any further reference.
%%%%%%%%%%%%%%%%%%%%%%%%%%%%%%%%%%%%%%%%%%%%%%%%%%%%%%%%%%%%%%%%%%%%%%%%%%%%%%%%%%%%%%%%%%%%%%

\section{Almost simple groups}

The main purpose of this section is to classify all almost simple groups whose prime graphs have no
triangles. We first consider the nonabelian simple groups. Our proof is based on the classification
of prime graphs of simple groups by D. White in \cite{White04,White06,White08}.
\begin{lemma}\label{Graphs of simple groups without triangles}
Let $S$ be a nonabelian simple group. Suppose that $\Delta(S)$ has no triangles. Then the following
hold.
\begin{enumerate}
\item $S\cong \PSL_2(2^f),$ where $|\pi(2^f\pm 1)|\leq 2$ and so $|\pi(S)|\leq 5.$
\item $S\cong \PSL_2(q),$ where $q=p^f$ is odd and $|\pi(q\pm 1)|\leq 2$ and so $|\pi(S)|\leq 4.$
\end{enumerate}
\end{lemma}

\begin{proof}
As $S$ is nonabelian simple, we obtain that $\rho(S)=\pi(S).$ By Burnside's $p^aq^b$ Theorem
\cite[Theorem~3.10]{Isaacs}, we deduce that $|\pi(S)|\geq 3.$ Assume first that $S\cong \PSL_2(q),$
with $q\geq 4$ a prime power. Since $\PSL_2(4)\cong\PSL_2(5),$ we will assume that $q>5$ when $q$
is odd. We have $|S|=q(q-1)(q+1)/\gcd(2,q-1).$ If $q$ is even, then $\cd(S)=\{1,q-1,q,q+1\};$ and
if $q> 5$ is odd, then $\cd(S)=\{1,(q+\epsilon)/2,q-1,q,q+1\},$ where $\epsilon=(-1)^{(q-1)/2}.$
(See, for example \cite{Lewis-White10}.) Since $\Delta(S)$ has no triangles, $|\pi(a)|\leq 2$ for
any character degree $a\in\cd(S).$ If $q>5$ is odd, then $q^2-1$ has at most two odd prime divisors
different from $p,$ and thus $|\pi(\PSL_2(q))|\leq 4.$ If $q$ is even, then it is clear that
$|\pi(S)|\leq 5$ since $|\pi(q\pm 1)|\leq 2.$ Now assume that $S\not\cong \PSL_2(q).$ By
\cite[Corollary~1.2]{White08} either $\Delta(S)$ is complete or the following cases hold.

\begin{enumerate}
\item[(i)] $S\in\{\rm{J}_1,\rm{M}_{11},\rm{M}_{23},\rm{A}_8\};$
\item[(ii)] $S\cong {}^2\rm{B}_2(q^2),$ where $q^2=2^{2m+1}$ and $m\geq 1;$
\item[(iii)] $S\cong \PSL^\epsilon_3(q),$ for some prime power $q>2$ and $\epsilon=\pm.$
\end{enumerate}

Obviously, if $\Delta(S)$ is complete, then it has a triangle as $|\pi(S)|\geq 3.$ Hence we can
assume that $\Delta(S)$ is not complete. Now the character tables of those groups listed in Case
$(\textrm{i})$ above can be found in \cite{ATLAS} and it is easy to check that the prime graphs of
these groups always contain a triangle. Assume that $S\cong {}^2{\rm{B}}_2(q^2)$ with
$q^2=2^{2m+1}$ and $m\geq 1.$ By \cite[Theorem~3.3]{White04}, we have that $\pi(S)=\{2\}\cup
\pi(q^2-1)\cup \pi(q^4+1),$ where the subgraph of $\Delta(S)$ on $\pi(S)-\{2\}$ is complete.
Clearly, $|\pi(S)|\geq 4$ and thus $|\pi(S)-\{2\}|\geq 3.$ Therefore, $\Delta(S)$ possesses a
triangle since $\pi(S)-\{2\}$ is complete with at least $3$ vertices. Hence this case cannot
happen. Finally, assume that $S\cong \PSL^\epsilon_3(q)$ with $q=p^f>2$ and $\epsilon=\pm.$ Assume
first that $q=4.$ Using \cite{ATLAS} we see that $\Delta(\PSL^\epsilon_3(4))$ has at least one
triangle. Hence these cases cannot happen.  Assume now that $q\neq 4.$ It follows from
\cite[Theorems~3.2~and~3.4]{White06} that $\pi(S)=\{p\}\cup \pi((q^2-1)(q^2+\epsilon q+1)),$ where
the subgraph of $\Delta(S)$ on $\pi(S)-\{p\}$ is complete. If $|\pi(S)|\geq 4,$ then
$|\pi(S)-\{p\}|\geq 3$ and thus $\Delta(S)$ possesses a triangle, a contradiction. Thus $|\pi(S)|=
3.$ By \cite[Table~1]{Huppert-Lempken}, we deduce that $q=3.$ Using \cite{ATLAS}, we can check that
$\Delta(S)$ is a triangle in both cases. Therefore, these cases cannot happen either. The proof is
now complete.
\end{proof}
For almost simple groups, we obtain the following result.

\begin{lemma}\label{Graphs of almost simple groups without triangles} Let $S$ be a nonabelian
simple group and let $G$ be an almost simple group with socle $S.$ Suppose that $\Delta(G)$ has no
triangles. Then $S\cong \PSL_2(q)$ with $q=p^f\geq 4$ a prime power, $\pi(G)=\pi(S)$ and
$|\pi(G)|\leq 5.$ Furthermore, if $|\pi(G)|=5,$ then $G=S\cong \PSL_2(2^f),$ where $f\geq 6$ and
$|\pi(2^f\pm 1)|=2.$
\end{lemma}

\begin{proof}
As $S\unlhd G,$ $\pi(S)\subseteq\pi(G)$ and $\Delta(S)$ is a subgraph of $\Delta(G).$ Thus
$\Delta(S)$ has no triangles and so every degree of both $S$ and $G$ has at most two distinct prime
divisors. By Lemma~\ref{Graphs of simple groups without triangles}, we obtain that $S\cong
\PSL_2(q),$ where $q=p^f\geq 4$ and $|\pi(q\pm 1)|\leq 2.$ Since $\PSL_2(4)\cong\PSL_2(5),$ we can
assume that $q\neq 5$ and $q\geq 4.$ Suppose by contradiction that $\pi(S)\neq \pi(G)$ and let
$r\in\pi(G)-\pi(S).$ It follows that $r$ divides $m:=|G:G\cap \PGL_2(q)|$ since $|G\cap
\PGL_2(q):S|\leq 2.$ Hence $r$ must divide $f$ and $r\not\in\pi(S),$ so $f\geq 5$ and $q>9.$ By
\cite[Lemma~4.5]{Lewis-White10}, we have $m(q\pm 1)\in\cd(G).$ If either $|\pi(q-1)|=2$ or
$|\pi(q+1)|=2,$ then $m(q-1)$ or $m(q+1)$ is divisible by three distinct primes, a contradiction.
Hence both $q\pm 1$ are prime powers. We deduce that $q$ is even and we obtain that $q=4$ or $q=8,$
which is a contradiction as $q>9.$ Therefore, we always have that $\pi(G)=\pi(S).$ By applying
Lemma~\ref{Graphs of simple groups without triangles}, we deduce that $|\pi(G)|=|\pi(S)|\leq 4$ for
odd $q$ and $|\pi(G)|=|\pi(S)|\leq 5$ for even $q.$ In particular, $|\pi(G)|=|\pi(S)|\leq 5$ for
all $q.$ Now assume that $|\pi(S)|=|\pi(G)|=5.$ By the argument above, we must have that $q=2^f,$
$|\pi(2^f\pm 1)|=2$ and so $f\geq 6.$ We now claim that $G=S.$ Suppose by contradiction that $S\neq
G.$ In particular, $|G:S|$ is divisible by some prime $r.$ By invoking
\cite[Lemma~4.5]{Lewis-White10} again, we have that both $|G:S|(q\pm 1)$ are degrees of $G.$ If
$r\not\in\pi(q^2-1),$ then $|G:S|(q-1)$ is divisible by three distinct primes, a contradiction.
Thus $r\in\pi(q^2-1).$ Since $\gcd(q-1,q+1)=1,$ we deduce that $r\in \pi(q-\delta)$ where
$\delta=1$ or $-1.$ But then $|G:S|(q+\delta)$ is divisible by three distinct primes since
$r\not\in\pi(q+\delta).$ This contradiction shows that $G=S$ as required.
\end{proof}

%%%%%%%%%%%%%%%%%%%%%%%%%%%%%%%%%%%%%%%%%%%%%%%%%%%%
\section{Prime graphs of nonsolvable groups}

In this section, we prove several auxiliary lemmas which will be needed in the proofs of our main
results. Some special cases of the main theorems are treated here. The proof of the first part of
the next lemma is similar to that of \cite[Lemma~3.1]{Lewis-White12}. For completeness, we
reproduce the proof here. Recall that the solvable radical of a group $G$ is the largest solvable
normal subgroup of $G.$

\begin{lemma}\label{Reduction to almost simple groups}
Let $G$ be a nonsolvable group and let $N$ be the solvable radical of $G.$ Suppose that $\Delta(G)$
has no triangles. Then there exists a normal subgroup $M$ of $G$ such that $M/N\cong \PSL_2(q)$
with $q\geq 4$ a prime power, and $G/N$ is an almost simple groups with socle $M/N.$ Furthermore,
$\rho(M)=\rho(G).$
\end{lemma}

\begin{proof}
Let $N$ be the solvable radical of $G$ and let $M$ be a normal subgroup of $G$ such that $M/N$ is a
chief factor of $G.$ Then $M/N$ is nonsolvable and so $M/N\cong S^k,$ where $S$ is a nonabelian
simple group and $k\geq 1$ is an integer. Since $|\pi(S)|\geq 3,$ if $k\geq 2,$ then $\Delta(M/N)$
is a complete graph and thus possesses a triangle, which is a contradiction as $\Delta(M/N)$ is a
subgraph of $\Delta(G).$ Hence we conclude that $M/N\cong S$ is a nonabelian simple group. Now let
$C/N=\Centralizer_{G/N}(M/N).$ Then $N\leq C\unlhd G$ and $M\cap C=N$ as $M/N$ is nonabelian
simple. Assume that $N\neq C.$ Then $C$ is nonsolvable and we can find a normal subgroup $N\leq
L\leq C$ such that $L/N$ is a nonabelian chief factor of $G.$ With the same reasoning as above, we
obtain that $L/N$ is a nonabelian simple group. Observe that every vertex in $\pi(L/N)\cap
\pi(M/N)$  will be adjacent to all of the vertices in $\pi(L/N)\cup \pi(M/N).$ Thus $\pi(L/N)\cap
\pi(M/N)$ induces a complete subgraph of $\Delta(G)$ and so $|\pi(L/N)\cap \pi(M/N)|\leq 2.$ By
Lemma~\ref{Graphs of almost simple groups without triangles}, we know that each of $L/N$ and $M/N$
is isomorphic to $\PSL_2(q)$ for possibly different $q.$ In particular, $\{2,3\}\subseteq
\pi(L/N)\cap \pi(M/N),$ therefore $\rho(L/N)\cap \rho(M/N)=\{2,3\}.$ Since $|\pi(M/N)|\geq 3,$
there exists a prime $r\in \pi(M/N)$ such that $r>3.$ But then $\{2,3,r\}$ will form a triangle in
$\Delta(G),$ a contradiction. Hence $C=N$ and so $G/N$ is an almost simple group with socle
$M/N\cong\PSL_2(q)$ for some prime power $q\geq 4.$

We next claim that $\rho(G)=\rho(M).$ Since $M\unlhd G,$ every character degree of $M$ must divide
some character degree of $G$ and so $\rho(M)\subseteq\rho(G).$ For the other inclusion, let
$r\in\rho(G).$ Then there exists $\chi\in\Irr(G)$ with $r$ dividing $\chi(1).$ Let
$\theta\in\Irr(N)$ be an irreducible character of $\chi_N.$ As $r$ divides $\chi(1),$ we deduce
that $r$ divides either $\chi(1)/\theta(1)$ or $\theta(1).$ The first possibility implies that
$r\in\pi(G/N)$ since $\chi(1)/\theta(1)$ divides $|G/N|$ by \cite[Corollary~11.29]{Isaacs}. The
latter implies that $r\in\rho(N).$ Thus $r\in \pi(G/N)\cup\rho(N).$ As $\pi(G/N)=\pi(M/N)$ by
Lemma~\ref{Graphs of almost simple groups without triangles}, we deduce that $r\in\pi(M/N)\cup
\rho(N).$ In particular, as $\pi(M/N)=\rho(M/N)\subseteq \rho(M)$ and $\rho(N)\subseteq \rho(M),$
we deduce that $r\in\rho(M).$ Thus $\rho(G)\subseteq\rho(M).$ Therefore, $\rho(G)=\rho(M)$ as
required.
\end{proof}
%%%%%%%%%%%%%%%%%%%%%%%%%%%%%%%%%%%%%%%%%%%%%%%%%%%%%%%%%%%%%%%%%%%%%%%%%%%%%%%%%%%%%%%%%%%%%%%%%%
The following important result will be used frequently. The proof of this lemma makes use of the
classification of subgroups of prime power index of simple groups in the form of Lemma~\ref{power
index and solvable} and a result due to Higgs (see \cite[Theorem~2.3]{Moreto}).
\begin{lemma}\label{Extendible characters}
Let $N$ be a normal subgroup of a group $G$ such that $G/N\cong S,$ where $S$ is a nonabelian
simple groups. Let $\theta\in\Irr(N).$ Then either $\chi(1)/\theta(1)$ is divisible by two distinct
primes in $\pi(G/N)$ for some $\chi\in\Irr(G|\theta)$ or $\theta$ is extendible to
$\theta_0\in\Irr(G)$ and $G/N\cong \rm{A}_5$ or $\PSL_2(8).$
\end{lemma}

\begin{proof}

Let $\theta\in\Irr(N).$ By \cite[Corollary~11.29]{Isaacs}, for any $\chi\in\Irr(G|\theta),$ we have
that $\chi(1)/\theta(1)$ divides $|G/N|.$ Thus if $\chi(1)/\theta(1)$ is divisible by two distinct
primes then it is also divisible by two primes in $\pi(G/N).$ Hence we can assume that
$\chi(1)/\theta(1)$ is a prime power for any $\chi\in\Irr(G|\theta)$ and we will show that $\theta$
is extendible to $G$ and that $G/N\cong \rm{A}_5$ or $\PSL_2(8).$

We now claim that $\theta$ is $G$-invariant. Suppose the contrary. Let $I=I_G(\theta).$ Then
$N\unlhd I$ and $I/N$ is a proper subgroup of $G/N.$ Writing $\theta^I=\sum_{i=1}^m e_i\phi_i,$
where $\phi_i\in\Irr(I|\theta)$ and $m\geq 1.$ Then for each $i,$ we have that
$\phi_i^G\in\Irr(G|\theta)$ and $\phi_i^G(1)=|G:I|e_i\theta(1)\in\cd(G).$ Therefore, $|G:I|e_i$ is
a prime power for all $i$ with $1\leq i\leq m.$ In particular, $|G:I|=r^a,$ where $r$ is a prime
and $a\geq 1.$ By Lemma~\ref{power index and solvable}, we have that either $I/N$ is nonsolvable or
$I/N$ is a nonabelian Hall subgroup of $G/N.$ Assume that the latter case holds. Since $I/N$ is
nonabelian, there exists $j$ with $1\leq j\leq m$ such that $e_j>1.$ Since $e_j$ divides $|I/N|,$
where $\gcd(|I/N|,|G:I|)=1,$ we deduce that $e_j$ is prime to $r^a.$ Thus
$\phi_j^G(1)/\theta(1)=r^ae_j$ is divisible by at least two distinct primes, a contradiction.
Assume now that the former case holds. Then $I/N$ is a nonsolvable group and $\theta$ is
$I$-invariant. It follows from \cite[Theorem~2.3]{Moreto} that $\phi_k(1)/\theta(1)$ is not an
$r$-power for some $k$ with $1\leq k\leq m.$ Hence $\phi_k(1)/\theta(1)$ is divisible by some prime
$s\neq r$ and so $\phi_k^G(1)/\theta(1)$ is divisible by two distinct primes. This contradiction
shows that $\theta$ must be $G$-invariant.

Now assume that $\theta$ is $G$-invariant but not extendible to $G.$ Then $(G,N,\theta)$ is
character triple isomorphic to the triple $(L,A,\lambda)$ by \cite[Chapter~11]{Isaacs}, where $L$
is perfect, $A\leq \textrm{Z}(L),$ $L/A\cong G/N$ and $\lambda\in\Irr(A)$ is nontrivial. Then for
any $\chi\in\Irr(L|\lambda),$ we have that $\chi(1)/\lambda(1)=\chi(1)$ is a nontrivial prime
power. Since $\lambda$ is nontrivial, we obtain that $o(\lambda)$ is nontrivial and so $p\mid
o(\lambda)$ for some prime $p.$ By \cite[Lemma~2.1]{Moreto}, we have that $p\mid\chi(1)$ for all
$\chi\in\Irr(L|\lambda)$ and thus $\chi(1)$ is a nontrivial $p$-power for all
$\chi\in\Irr(L|\lambda).$ Now \cite[Theorem~2.3]{Moreto} yields that $L/A\cong G/N$ is solvable,
which is impossible.

Finally, assume that $\theta$ is extendible to $\theta_0\in\Irr(G).$ By Gallagher's Theorem,
$\theta_0\psi$  are all the irreducible constituents of $\theta^G,$ where $\psi\in\Irr(G/N).$
Therefore,  $\theta_0(1)\psi(1)/\theta(1)=\psi(1)$ is a prime power for all $\psi\in\Irr(G/N).$ By
\cite[Corollary]{Manz-Staszewski-Willems}, we obtain that $G/N\cong \textrm{A}_5$ or $\PSL_2(8).$
\end{proof}

\begin{lemma}\label{Quotient} Let  $N$ be a
normal subgroup of a group $G$ such that $G/N\cong \PSL_2(2^f)$ and $|\rho(G)|=|\pi(G/N)|=5.$ If
$\Delta(G)$ has at most two connected components, then $\Delta(G)$ contains a triangle.
\end{lemma}

\begin{proof} Suppose that $\Delta(G)$ has at most two connected components
and contains no triangles. As $|\pi(G/N)|=5$ and $\Delta(G/N)$ has no triangles, we deduce that
$|\pi(2^f-1)|=|\pi(2^f+1)|=2$ and $f\geq 6.$ We have that
\begin{equation}\label{eqn1}\cd(\PSL_2(2^f))=\{1,2^f-1,2^f,2^f+1\}.
\end{equation}
Therefore $\Delta(G/N)$  has three connected components $$\{2\},\pi(2^f-1)
~\mbox{and}~ \pi(2^f+1).$$ We now consider the case when $\Delta(G)$ is connected and disconnected
separately.

{\bf Case} $\Delta(G)$ is connected. Then the vertex $2$ is adjacent to some vertex $r,$  where
$r\in\pi(2^f-1)$ or $r\in\pi(2^f+1).$ In particular, $r$ is odd. Hence there exists
$\chi\in\Irr(G)$ with $\pi(\chi(1))=\{2,r\}.$ Let $\theta\in\Irr(N)$ be any irreducible constituent
of $\chi_N.$ Then $\theta$ is nontrivial as $2$ is an isolated vertex in $\Delta(G/N).$

Assume first that $\theta$ is $G$-invariant. Since $f\geq 6,$ we deduce that the Schur multiplier
of $G/N\cong \PSL_2(2^f)$ is trivial and thus by Lemma~\ref{Extension}, $\theta$ is extendible to
$\theta_0\in\Irr(G).$ Now by Lemma~\ref{Gallagher theorem}, we have that
$\Irr(G|\theta)=\{\theta_0\lambda|\lambda\in\Irr(G/N)\}.$
 In particular, $\chi=\theta_0\mu$ for some
$\mu\in\Irr(G/N)$ as $\chi\in\Irr(G|\theta).$ As $\mu\in\Irr(G/N),$ we deduce that
$\mu(1)\in\{1,2^f,2^f\pm 1\}.$ If $\mu(1)=1$ or $2^f,$ then $r$ must divide $\theta_0(1)$ because
$\pi(\chi(1))=\pi(\theta(1))\cup\pi(\mu(1)).$ Now if $r\in\pi(2^f+1),$ then by taking
$\gamma\in\Irr(G/N)$ with $\gamma(1)=2^f-1,$ we have that $\theta_0(1)\gamma(1)\in\cd(G)$ is
divisible by three distinct primes, which is a contradiction. The case when $r\in\pi(2^f-1)$ can be
argued similarly. Therefore $\mu(1)=2^f-1$ or $2^f+1.$ As both $2^f-1$ and $2^f+1$ are divisible by
two distinct odd primes and that $\mu(1)\mid\chi(1),$ we deduce that $\chi(1)$ is divisible by two
distinct odd primes, which is impossible as $\pi(\chi(1))=\{2,r\}.$

Assume next that $\theta$ is not $G$-invariant. Then $N\leq I=I_G(\theta)< G.$ Let $K$ be a
subgroup of $G$ such that $I/N\leq K/N$ and that $K/N$ is a maximal subgroup of $G/N\cong
\PSL_2(q)$ with $q=2^f\geq 2^6.$ Writing $\theta^I=\sum_{i=1}^m e_i\phi_i,$ where  $e_i\geq 1$ and
$\phi_i\in\Irr(I|\theta)$ for all $1\leq i\leq m.$ By \cite[Theorem~19.6]{Huppert}, for each $i,$
we have that
 $\phi_i^G(1)=|G:I|e_i\theta(1)\in\cd(G).$ Since $\chi\in\Irr(G|\theta),$ we have that
$\chi=\phi_j^G$ for some $j$ with $1\leq j\leq m.$ Hence $\chi(1)=|G:I|e_j\theta(1)=2^cr^d,$ where
$c,d\geq 1$ are integers. Notice that $|G:I|$ is divisible by $|G:K|=|G/N:K/N|,$ which is the index
of a maximal subgroup of $G/N\cong \PSL_2(2^f).$ Thus $\pi(|G:K|)\subseteq \{2,r\},$ which means
that $|G:K|$ is divisible by at most one odd prime. Checking the list of maximal subgroups of
$\PSL_2(2^f)$ in \cite[Hauptsatz II.8.27]{Huppert67}, the index $|G:K|$ is one of the following
numbers:
\begin{equation}\label{eqn3} 2^{f-1}(2^f+1), 2^{f-1}(2^f-1), 2^f+1,\frac{2^f(2^{2f}-1)}{2^a(2^{2a}-1)},\end{equation}
 where $f/a=n\geq 2$
is a prime. Since $f\geq 6$ and both $2^f-1$ and $2^f+1$ have two distinct odd prime divisors, we
can see that the first three possibilities for $|G:K|$ cannot happen. Finally, assume that $|G:K|$
takes the last value in \eqref{eqn3}. Since
 $(2^{2f}-1)/(2^{2a}-1)>1$ is odd and $\pi(|G:K|)\subseteq\{2,r\},$ we must have that
\begin{equation}\label{eqn2}\frac{2^{2f}-1}{2^{2a}-1}=r^k\end{equation} for some integer $k\geq 1.$
If $n=2,$ then ${(2^{2f}-1)}/{(2^{2a}-1)}=2^f+1=r^k.$ But this would contradict our assumption that
$2^f+1$ has two distinct prime divisors. Hence $n>2$ is a prime. But then Lemma~\ref{2-powers}
shows that equation \eqref{eqn2} cannot happen. Thus this case cannot happen.

{\bf Case} $\Delta(G)$ has two connected components. By \cite[Theorem~6.3]{Lewis-White03}, the
smaller connected component of $\Delta(G)$ has only one vertex. Thus this vertex must be $2.$ Hence
$\pi(2^f-1)$ and $\pi(2^f+1)$ must lie in the same connected component and so there exists
$\chi\in\Irr(G)$ such that $\pi(\chi(1))=\{u,v\},$ where $u\in\pi(2^f-1)$ and $v\in\pi(2^f+1).$
Since $|\pi(2^f\pm 1)|=2$ and $\gcd(2^f-1,2^f+1)=1,$ we have that $\pi(2^f-1)=\{u,r\}$ and
$\pi(2^f+1)=\{v,s\},$ where $\{u,r\}\cap \{v,s\}$ is empty. It follows that $r$ and $s$ do not
divide $\chi(1),$ and so neither $2^f+1$ nor $2^f-1$ can divide $\chi(1).$ Let $\theta\in\Irr(N)$
be an irreducible constituent of $\chi_N.$ Clearly, $\theta$ is not the principal character of $N.$
Assume first that $\theta$ is not $G$-invariant and let $I=I_G(\theta).$ Then $I/N$ is a proper
subgroup of $G/N\cong\PSL_2(2^f)$ and thus $|G:I|$ is divisible by the index of some maximal
subgroup of $G.$ Furthermore, this index must be odd as $|G:I|$ divides $\chi(1).$ By \eqref{eqn3},
we obtain that $2^f+1$ divides $|G:I|$ and so divides $\chi(1),$ a contradiction. Thus $\theta$ is
$G$-invariant. As the Schur multiplier of $G/N\cong \PSL_2(2^f)$ with $f\geq 6,$ is trivial, we
deduce from Lemma~\ref{Extension} that $\theta$ extends to $\theta_0\in\Irr(G).$ Now Gallagher's
Theorem yields that $\Irr(G|\theta)=\{\theta_0\psi|\psi\in\Irr(G/N)\}.$ Since
$\chi\in\Irr(G|\theta),$ we deduce that $\chi=\theta_0\mu$ for some $\mu\in\Irr(G/N)$ and so
$\mu(1)\in\{1,2^f,2^f\pm 1\}$ and $\mu(1)\mid\chi(1).$ As $\chi(1)$ is odd and $2^f\pm 1\nmid
\chi(1),$ we deduce that $\mu(1)=1.$ Thus $\chi(1)=\theta_0(1).$ By taking $\gamma\in\Irr(G/N)$
with $\gamma(1)=2^f-1,$ we deduce that $\theta_0(1)\gamma(1)=\chi(1)(2^f-1)\in\cd(G),$ which is
impossible as this degree is divisible by three distinct primes $u,v$ and $r.$ Therefore,
$\Delta(G)$ always contains a triangle.
\end{proof}

%%%%%%%%%%%%%%%%%%%%%%%%%%%%%%%%%%%%%%%%%%%%%%%%%%%%%%%%%%%%%%%%%%%%%%%%%%%%%%%%%%%%%%%%%%%%%%%%%%%%

In the last result of this section, we prove the following technical result.
\begin{lemma}\label{Invariant characters}
Let  $N\unlhd G$ be a solvable subgroup of $G.$  Suppose that $G/N$ is a nonabelian simple group
such that $\Delta(G)$ has no triangles. Let $\tau=\rho(G)-\pi(G/N).$ Then the following hold.
\begin{enumerate}
\item $\tau\subseteq \rho(N),$ there is no edges among primes in $\tau$ and  $|\tau|\leq 2.$
\item If $\tau\neq \emptyset,$ then for any $r\in\tau$ and  $\theta\in\Irr(N)$ with $r\mid\theta(1),$
$\theta$ is extendible to $\theta_0\in\Irr(G)$ and $G/N\cong \rm{A}_5$ or $\PSL_2(8).$
\end{enumerate}
\end{lemma}

\begin{proof}
Observe that $|\pi(\psi(1))|\leq 2$ for any $\psi\in\Irr(G)$ since $\Delta(G)$ has no triangles.

For $(1),$ if $\tau$ is empty, then there is nothing to prove. Hence we assume that $\tau$ is
nonempty. Let $r\in \tau.$ Since $r\in\rho(G)-\pi(G/N),$ there exists $\chi\in\Irr(G)$ with $r\mid
\chi(1)$ but $r\not\in\pi(G/N).$ Let $\theta\in\Irr(N)$ be any irreducible constituent of $\chi_N.$
By \cite[Corollary~11.29]{Isaacs}, we know that $\chi(1)/\theta(1)$ divides $|G/N|.$ As $r$ is
prime to $|G/N|,$ we deduce that $r$ is prime to $ \chi(1)/\theta(1)$ and thus $r\mid\theta(1),$
which means that $r\in\rho(N).$ Since $r$ is chosen arbitrarily in $\tau,$ we deduce that
$\tau\subseteq\rho(N).$ We next show that there is no edges among primes in $\tau.$ If $|\tau|\leq
1,$ then the result is clear. Now assume that there exist two distinct primes in $\tau,$ say $r<s,$
such that $r$ is adjacent to $s$ via $\chi\in\Irr(G).$ Since $|\pi(\chi(1))|\leq 2,$  we obtain
that $\pi(\chi(1))=\{r,s\}.$ Since $\{r,s\}\subseteq\tau=\rho(G)-\pi(G/N),$ both $r$ and $s$ are
prime to $|G/N|.$ Therefore, $\gcd(\chi(1),|G/N|)=1,$  so  $2<r<s.$ Thus by
\cite[Theorem~21.3]{Huppert}, we have that $\chi_N\in \Irr(N).$ By Lemma~\ref{Gallagher theorem},
we obtain that $\chi(1)\lambda(1)\in\cd(G)$ for all $\lambda\in\Irr(G/N).$ By taking any
$\mu\in\Irr(G/N)$ with $2\mid \mu(1),$ we see that $\chi(1)\mu(1)\in\cd(G)$ is divisible by three
distinct primes $2,r$ and $s$, which is a contradiction. Hence there is no edges among vertices in
$\tau.$ As $\tau\subseteq \rho(N),$ where $N$ is solvable by hypothesis and there is no edges in
$\tau$ by the previous claim, the supposition $|\tau|\geq 3$ would violate Lemma~\ref{Palfy
condition}. Hence $|\tau|\leq 2$ as wanted. This completes the proof of $(1).$

For $(2),$ let $r\in\tau$ and let $\theta\in\Irr(N)$ such that $r$ divides $\theta(1).$ By
Lemma~\ref{Extendible characters}, either $\chi(1)/\theta(1)$ is divisible by two distinct primes
in $\pi(G/N)$ for some $\chi\in\Irr(G|\theta)$ or $\theta$ is extendible to $\theta_0\in\Irr(G)$
and $G/N\cong \textrm{A}_5$ or $\PSL_2(8).$ If the first case holds, then since $r\nmid
\chi(1)/\theta(1),$ we deduce that $\chi(1)$ has at least three distinct prime, which contradicts
the fact that $|\pi(\chi(1))|\leq 2.$ So, this case cannot happen and thus the latter case holds
and this completes the proof of the lemma.

\end{proof}

%%%%%%%%%%%%%%%%%%%%%%%%%%%%%%%%%%%%%%%%%%%%%%%%%%%%%%%%%%%%%%%%%%%%%%%%%%%%%%%%%%%%%%%%%%%%%%%%

\section{Proofs of main theorems}

We are now ready to prove our main results.
\begin{proof}[\bf Proof of Theorem~A] Assume that $G$ is a group whose prime graph $\Delta(G)$ has no triangles.
If $G$ is solvable, then $|\rho(G)|\leq 4$ by Lemma~\ref{Square or Triangle}. Thus we can assume
that $G$ is nonsolvable. Let $N$ be the solvable radical of $G.$ By Lemma~\ref{Reduction to almost
simple groups}, there exists a normal subgroup $M$ of $G$ such that $G/N$ is an almost simple group
with socle $M/N\cong\PSL_2(q)$ and $|\rho(G)|=|\rho(M)|,$ where $q\geq 4$ is a prime power. Since
$M\unlhd G,$ $\Delta(M)$ is a subgraph of $\Delta(G),$ so $\Delta(M)$ contains no triangles. By
Lemma~\ref{Graphs of simple groups without triangles}, we have that $|\pi(M/N)|\leq 5.$ Let
$\tau=\rho(M)-\pi(M/N).$ By Lemma~\ref{Invariant characters}$(1),$ we obtain that $|\tau|\leq 2,$
and if $\tau$ is nonempty, then $G/N\cong \textrm{A}_5$ or $\PSL_2(8).$ Assume first that $\tau$ is
empty. Then $|\rho(G)|=|\rho(M)|=|\pi(M/N)|\leq 5$ and we are done. So, we can assume that
$\tau\neq\emptyset.$ It follows that $G/N\cong A_5$ or $\PSL_2(8).$ In particular, $|\pi(G/N)|=3$
and thus $|\rho(G)|=|\rho(M)|=|\tau|+|\pi(G/N)|\leq 2+3=5.$ Thus $|\rho(G)|\leq 5$ in both cases.
The proof is now complete.
\end{proof}
%%%%%%%%%%%%%%%%%%%%%%%%%%%%%%%%%%%%%%%%%%%%%%%%%%%%%%%%%%%%%%%%%%%%%%%%%%%%%%%%%%%%%%%%%%%%%%%%%

\begin{proof}[\bf Proof of Theorem~B] Let $G$ be a counterexample with minimal order.
Then $|\rho(G)|=5$ and $\Delta(G)$ has no triangles but $G$ does not satisfy the conclusions of the
theorem. First of all, by Lemma~\ref{Square or Triangle} we can assume that $G$ is nonsolvable.
Furthermore, if $\Delta(G)$ has three connected components, then conclusion $(1)$ holds by
\cite[Theorem~4.1]{Lewis-White03}. Thus we can assume that $\Delta(G)$ has at most two connected
components. Let $N$ be the solvable radical of $G.$ By Lemma~\ref{Reduction to almost simple
groups}, there exists a normal subgroup $N\unlhd M\unlhd G$ such that $G/N$ is an almost simple
group with socle $M/N,$ where $M/N\cong \PSL_2(q)$ with $q\geq 4$ being a power of a prime $p.$
Furthermore, $\rho(G)=\rho(M).$ As $N\unlhd G$ and $\rho(G)=\rho(M),$ we deduce that $\Delta(M)$ is
a subgraph of $\Delta(G)$ with the same set of vertices. Notice that  $M$ is nonsolvable and $N$ is
also the solvable radical of $M.$

\begin{step}\label{StepA1} $M=G.$

Suppose the contrary. Then $M$ is a proper subgroup of $G$ and $\Delta(M)$ has no triangles with
$|\rho(M)|=5.$ By the minimality of $|G|,$ we deduce that either $\Delta(M)$ is the second graph in
Figure A and $M\cong \PSL_2(2^f)\times A$ with $|\pi(2^f\pm 1)|=2$ and $A$ being abelian or
$\Delta(M)$ is the first graph in Figure A and $M\cong H\times K,$ where $H$ and $K$ satisfy the
conditions in Case $(2)$ of Theorem~B.

Assume that the latter case holds. Since $\Delta(M)$ and $\Delta(G)$ have the same set of vertices,
we deduce that $\Delta(G)$ is obtained from $\Delta(M)$ by adding some edges. However no more edges
can be added to $\Delta(M)$ without forming a triangle and since $\Delta(G)$ contains no triangles,
we deduce that $\Delta(G)=\Delta(M).$ As $K\unlhd M$ is solvable, we deduce that $K\leq N\unlhd M.$
As $M/K$ is nonabelian simple, we must have that $K=N$ and so $M/N\cong H,$ where
$H\cong\textrm{A}_5$ or $\PSL_2(8).$ Since $M=H\times K$ and $\Delta(M)=\Delta(G)$ is the first
graph in Figure A, we deduce that each prime in $\pi(H)$ is of degree $2$ and each prime in
$\rho(K)$ is of degree $3$ in $\Delta(G).$ Hence there is no edges among primes in $\pi(H).$ Since
$G/N$ is an almost simple group with socle $M/N\cong H$ and $|G/N:M/N|$ is nontrivial, we deduce
that $G/N\cong \textrm{A}_5\cdot 2$ or $\PSL_2(8)\cdot 3$ by using \cite{ATLAS}. However both cases
are impossible since $G/N$ always possesses a character degree which is divisible by two distinct
primes in $\pi(G/N)=\pi(H)$ and hence there is an edge among primes in $\pi(H),$ contradicting our
previous claim. Therefore, this case cannot happen.

Assume now that the former case holds. Since $A$ is an abelian normal subgroup of $M$ and $M/A$ is
nonabelian simple, with the same reasoning as in the previous paragraph, we deduce that $N=A$ and
so $M/N\cong \PSL_2(2^f)$ and $|\rho(M)|=|\pi(M/N)|=5.$ Hence, $|\pi(G/N)|=5$ as $|\pi(M/N)|\leq
|\pi(G/N)|\leq |\rho(G)|=5.$ Since $G/N$ is an almost simple group with socle $M/N,$ $|\pi(G/N)|=5$
and $\Delta(G/N)$ has no triangles, by Lemma~ \ref{Graphs of almost simple groups without
triangles} we deduce that $G/N=M/N$ and thus $G=M.$ However this contradicts our assumption that
$M$ is a proper subgroup of $G.$

Therefore $M=G$ as we wanted.
\end{step}

\begin{step}\label{StepA2} Let $\tau=\rho(G)-\pi(G/N).$ Then $G/N\cong \textrm{A}_5$ or $\PSL_2(8),$  $|\pi(G/N)|=3,$ $|\tau|=2$
 and if $r\in \tau$ and
$\theta\in\Irr(N)$ with $r\mid \theta(1),$ then $\theta$ is extendible to $\theta_0\in\Irr(G).$

From Step~\ref{StepA1}, we have that $G/N\cong \PSL_2(q)$ with $q\geq 4.$ Assume first that $\tau$
is empty. It follows that $\rho(G)=\pi(G/N),$ so $|\pi(G/N)|=|\rho(G)|=5.$ We deduce from
Lemma~\ref{Graphs of almost simple groups without triangles} that $G/N\cong \PSL_2(2^f)$ where
$f\geq 6$ and $|\pi(2^f\pm 1)|=2.$  But then since $\Delta(G)$ has at most two connected
components, Lemma~\ref{Quotient} yields a contradiction. Hence, we conclude that $\tau\neq\emptyset$
and thus $G/N\cong \textrm{A}_5$ or $\PSL_2(8)$ by Lemma~\ref{Invariant characters}$(2).$
Therefore, $|\pi(G/N)|=3$ and so $|\tau|=2$ since $|\rho(G)|=|\tau|+|\pi(G/N)|=5.$ The remaining
statement follows from Lemma~\ref{Invariant characters}$(2)$ since $\tau$ is nonempty.
\end{step}

Writing $\pi(G/N)=\{p_1,p_2,p_3\}$ and $\tau=\{r_1,r_2\}.$ Then $r_1\neq r_2,$ $p_i\neq p_j$ for
$1\leq i\neq j\leq 3$ and $\{p_1,p_2,p_3\}\cap\{r_1,r_2\}=\emptyset.$

\begin{step}\label{StepA3} For each $i=1,2,$ $r_i$ is of degree $3$ and for each $j=1,2,3,$ $p_j$ is of degree $2$ in
$\Delta(G).$ Hence, there is no edges among primes $p_j,1\leq j\leq 3,$ in the graph $\Delta(G).$

By Lemma~\ref{Invariant characters}$(1),$ we know that $\tau\subseteq\rho(N)$ and so for each
$i=1,2,$ there exists $\theta_i\in\Irr(N)$ with $r_i\mid\theta_i(1).$ Also, $r_1$ and $r_2$ are not
adjacent in $\Delta(G).$ By Step~\ref{StepA2}, both $\theta_i$ extend to $G$ and hence by applying
Gallagher's Theorem, since $\pi(G/N)=\{p_1,p_2,p_3\},$ we deduce that each vertex $r_i$ is
connected to all vertices $p_j,$ and so $r_i$ is of degree $3$ in $\Delta(G).$ Finally, there is no
edges among primes $p_j,1\leq j\leq 3,$ as otherwise $\Delta(G)$ would contain a triangle.
\end{step}

\begin{step}\label{StepA4} Let $H$ be the last term of the derived series of $G.$ Then $H\cong
\textrm{A}_5$ or $\PSL_2(8)$ and $G\cong H\times N.$

As $H$ is the last term of the derived series of $G,$ where $G$ is nonsolvable, we deduce that $H$
is perfect. Let $U=H\cap N.$ Then $U\unlhd G.$ Since $G/N$ is nonabelian simple and $HN/N\unlhd
G/N$ is a nontrivial normal subgroup, it must be that $G=HN.$ Hence $G/N\cong H/U.$ Thus $H/U$ is
isomorphic to either $\textrm{A}_5$ or $\PSL_2(8)$ by Step~\ref{StepA2}. It suffices to show that
$U$ is trivial. As $\Delta(H)$ is a subgraph of $\Delta(G),$ it follows from Step~\ref{StepA3} that
there is no edges among primes $\{p_1,p_2,p_3\}$ in the subgraph $\Delta(H).$

Suppose by contradiction that $U$ is nontrivial. Since $U$ is solvable, it has a nontrivial
character $\lambda\in\Irr(U)$ with $\lambda(1)=1.$ As $H/U$ is nonabelian simple and there is no
edges among primes in $\pi(H/U),$ we deduce from Lemma~\ref{Extendible characters} that $\lambda$
is extendible to $\lambda_0\in\Irr(H)$ and thus $H$ has a nontrivial linear character, which is
impossible as $H$ is perfect. Thus $U$ must be trivial and so $G\cong H\times N$ as required.
\end{step}

\begin{step} Completion of the proof.

We have proved that $G= H\times N,$ where $H$ is isomorphic to either $\textrm{A}_5$ or
$\PSL_2(8).$  Since $\rho(G)=\rho(H)\cup \rho(N)$ and $\rho(H)=\pi(G/N)=\{p_1,p_2,p_3\},$ we deduce
that $\rho(N)=\{r_1,r_2\}=\tau.$ Thus $\rho(H)\cap \rho(N)=\emptyset.$ As there is no edges among
primes in $\tau$ by Lemma~\ref{Invariant characters}$(1),$ we deduce that $\Delta(N)$  has two connected components. By taking $K=N,$ we see that $G=H\times K$ satisfies
$(2)$ of Theorem~B, which is a contradiction.
\end{step}
This final contradiction shows that $G$ must satisfy one of the conclusions in Theorem~B. The proof
of the theorem is now complete
\end{proof}

\begin{proof}[\bf Proof of Theorem C]
Suppose by contradiction that there exists a group $G$ whose prime graph is a cycle or a tree with
at least $n\geq 5$ vertices. As a cycle or a tree contains no triangles, we deduce from Theorem~A
that $n\leq 5.$ Thus $n=5.$ By Theorem~B, $\Delta(G)$ must be one of the graphs in Figure A. But
this is impossible as these graphs are neither a cycle nor a tree. This completes the proof.
\end{proof}

%%%%%%%%%%%%%%%%%%%%%%%%%%%%%%%%%%%%%%%%%%%%%%%%%%%%%%%%%%%%%%%%%%%%%%%%

\section*{Acknowledgment}
The author is grateful to M. Lewis for his help during the preparation of this work. The author would also like to thank the anonymous reviewer for the careful reading of the manuscript and for making numerous corrections and suggestions. As a result,  the exposition of our paper has been improved significantly.

%%%%%%%%%%%%%%%%%%%%%%%%%%%%%%%%%%%%%%%%%%%%%%%%%%%%%%%%%%%%%%%%%%%%%%%%%
%\bibliographystyle{amsplain}

\end{document}